\documentclass[leqno,12pt]{article}

\usepackage{amssymb, amsmath, amsthm, amscd}
\usepackage[all]{xy}

\theoremstyle{plain}
\newtheorem{theorem}{Theorem}[section]

\newtheorem{lemma}[theorem]{Lemma}
\newtheorem{corollary}[theorem]{Corollary}

\theoremstyle{remark}
\newtheorem*{remark}{\it Remark\/}
\newtheorem*{remarks}{\it Remarks\/}

\theoremstyle{definition}
\newtheorem{definition}[theorem]{Definition}

\newcommand{\AmX}{A^{<m_X}}

\newcommand{\Cf}{\textit{cf.}\;}

\newcommand{\cO}{\mathcal{O}}
\newcommand{\ee}{\boldsymbol{e}}

\newcommand{\gi}{g^{(i)}}
\newcommand{\gr}{\operatorname{gr}}
\newcommand{\grA}{\gr(A)}

\newcommand{\griA}{\gr^i(A)}
\newcommand{\griL}{\gr^i(L)}

\newcommand{\griLM}{\gr^i(L/M)}
\newcommand{\griM}{\gr^i(M)}

\newcommand{\griR}{\gr^i(R)}
\newcommand{\grL}{\gr(L)}

\newcommand{\grLM}{\gr(L/M)}
\newcommand{\grM}{\gr(M)}

\newcommand{\grR}{\gr(R)}

\newcommand{\Ie}{\textit{i.e.}\ }

\newcommand{\ideal}[1]{\langle{#1}\rangle}

\newcommand{\isom}{\overset{\simeq}{\longrightarrow}}
\newcommand{\IX}{\mathfrak{i}_X}

\newcommand{\lc}{\operatorname{lc}}
\renewcommand{\le}{\leqslant}
\newcommand{\length}{\mathrm{length}}
\renewcommand{\leq}{\leqslant}

\newcommand{\lp}{\operatorname{lp}}
\newcommand{\lt}{\operatorname{lt}}
\newcommand{\Lt}{\operatorname{Lt}}

\newcommand{\m}{\mathfrak{m}}

\newcommand{\minus}{\smallsetminus}
\renewcommand{\mod}{\,\mathrm{mod}\,}

\renewcommand{\O}{\cO}

\newcommand{\oplusa}{\oplus_{0\leq i<a}}
\newcommand{\oplusm}{\oplus_{0\leq i<m}}

\newcommand{\p}{\varpi}

\newcommand{\rank}{\mathrm{rank}}

\newcommand{\ssm}{\smallsetminus}

\newcommand{\suma}{\sum_{0\leq i<a}}

\newcommand{\TG}{T_G}
\newcommand{\toF}{\overset{F}{\to}}
\newcommand{\ToF}{\overset{F}{\Rrightarrow}}

\newcommand{\ToG}{\overset{G}{\Rrightarrow}}
\newcommand{\x}{\boldsymbol{x}}
\newcommand{\Xtil}{\widetilde{X}}
\newcommand{\XX}{\mathbb{X}}
\newcommand{\XXe}{\XX_{\ee}}
\newcommand{\Z}{\mathbb{Z}}
\newcommand{\zero}{\{0\}}
\def\sn{\smallskip\noindent}
\def\mn{\medskip\noindent}
\def\bn{\bigskip\noindent}
\begin{document}
\begin{flushright}
Version 2009.4.23
\end{flushright}
\begin{center}
{\huge 
Flat modules and Gr\"obner bases 
over truncated discrete valuation rings}

\vskip 1cm
{\Large 
Toshiro Hiranouchi%\footnote{%%%%%%%%%%%%%%%%%%%%%%%%%%% Support
%} 
and 
Yuichiro Taguchi%\footnote{%%%%%%%%%%%%%%%%%%%%%%%%%%%%% Support
%}
}
\end{center}
%%%%%%%%%%%%%%%%%%%%%%%%%%%%%%%%%%%%%%%%%%%%%%%%%%%%%%% Abstract
\begin{abstract}
We present basic properties of Gr\"obner bases of 
submodules of a free module of finite rank over a 
polynomial ring  $R$  with coefficients in a 
graded truncated discrete valuations ring  $A$. 
As an application, 
we give a criterion for a finitely generated $R$-module 
to be flat over  $A$.
Its non-graded version is also given. 
\end{abstract}

%%%%%%%%%%%%%%%%%%%%%%%%%%%%%%%%%%%%%%%%%%%%%%%%%%%%%% Introduction
\section{Introduction}
\label{sec:Introduction}
A truncated discrete valuation ring 
(abbreviated as {\it tdvr} in the following) 
is a commutative ring which is isomorphic to 
a quotient of finite length 
of a discrete valuation ring 
(equivalently, it can be defined to be 
an Artinian local ring whose maximal ideal is 
generated by one element, \Cf \cite{HT1}, Prop.\ 2.2). 
In this paper, 
we study Gr\"obner bases over graded tdvr's 
and their applications. 
In particular, we provide a flatness criterion for 
a module over a tdvr (Thm.\ \ref{thm:criterion}) 
in terms of certain numerical invariants 
calculated from a Gr\"obner basis for 
the associated graded module.
Such a study has been motivated by our work on 
the ramification theory of tdvr's 
(\cite{HT1}, \cite{HT2}). 

In Section \ref{sec:GB}, 
we recall (following \cite{AL94} and \cite{Lez08}) 
the general theory of Gr\"obner bases 
for submodules of a free module over a polynomial ring 
with an arbitrary Noetherian coefficient ring  $A$. 
It is refined in Section \ref{sec:flat} in the case where  
$A$  is a graded tdvr, and we obtain a 
flatness criterion for graded modules over a graded tdvr 
(Thm.\ \ref{thm:flat}). 
A similar criterion in the case of usual modules over a 
tdvr is obtained in Section \ref{sec:tdvr} 
by reducing to the graded case.

%%%%%%%%%%%%%%%%%%%%%%%%%%%%%%%%%%%%%%%%%%%%%%%%%%% Convention
Throughout this paper, 
all rings are commutative. 

\medskip\noindent
{\it Acknowledgments.} %%%%%%%%%%%%%%%%%%%%%%%% Acknowledgments
We thank the organizers, 
Jaebum Sohn and Hisao Taya, of the 
11th Japan-Korea Joint Seminar at Sendai 
for inviting us to the Seminar and 
providing us with the opportunity to write up 
the results in this paper. We thank also 
Takafumi Shibuta and one of the referees 
for pointing out some errors in the first version of 
this paper and suggesting remedy. Thanks are also due to  
Seidai Yasuda 
for his critical comments on our use of standard bases. 
%%%%%%%%%%%%%% Support
The first author has been partially supported by 
GCOE (Kyoto University)  
and 
JSPS KAKENHI \#21740015. 
The second author has been partially supported by 
JSPS KAKENHI \#19540036.
This work has also been supported by
JSPS Core-to-Core Program \#18005 (Coordinator: Makoto Matsumoto).

%%%%%%%%%%%%%%%%%%%%%%%%%%%%%%%%%%%%%%%%%%%%%%%%%%%%%%%%%%%%%%%%
%%%%%%%%%%%%%%%%%%%%%%%%%%%%%%%%%%%%%%%%%%%%%%%%%%%%%% Section 2
%%%%%%%%%%%%%%%%%%%%%%%%%%%%%%%%%%%%%%%%%%%%%%%%%%%%%%%%%%%%%%%%
\section{Gr\"obner bases over Noetherian rings}
\label{sec:GB}
We recall the theory of Gr\"obner bases for 
submodules of a free module of finite rank over a 
polynomial ring with coefficients in a 
Noetherian ring following Chapter 4 of \cite{AL94} 
and \cite{Lez08}. 

Let  
$A$  be a Noetherian ring and  
$R:=A[x_1,\ldots,x_n]$  the ring of polynomials in  
$n$  variables with coefficients in  $A$.
Let  $L$  be a free $R$-module of rank  $r\geq 1$. 
Fix an $R$-basis  $\ee=(e_1,\ldots,e_r)$  of  $L$.
Let 
\begin{align*}
  \XX  &:=\
  \{x_1^{m_1}\cdots x_n^{m_n}\ |\ 
	(m_1,\ldots,m_n)\in(\Z_{\ge 0})^n\}, \\
  \XXe &:=\ 
  \{x_1^{m_{l,1}}\cdots x_n^{m_{l,n}}e_l\ |\ 
	1\leq l\leq r,\ 
	(m_{l,1},...,m_{l,n})\in(\Z_{\ge 0})^n\}
\end{align*}
be the sets of all {\it power products} (or {\it monomials}) 
in  $R$  and  $L$, respectively.
Choose and fix a {\it term order}\;  $<$  on  $\XXe$; 
thus it is a total order on  $\XXe$  satisfying 
\\
(1) 
$X<\x X$  for 
any  $X\in\XXe$ and 
any  $\x\neq 1$  in  $\XX$; 
\\
(2) 
If     $X< Y$  in  $\XXe$, 
then  $\x X<\x Y$  for any  $\x\in\XX$.

\sn
It is known (\cite{Lez08}, Prop.\ 3) that 
any term order makes  $\XXe$  a well-ordered set. 
For  $X,Y\in\XXe$, we write   
$X\mid Y$  if there exists  $\x\in\XX$  such that  $Y=\x X$.

Any non-zero element  $f$  of  $L$  can be written uniquely as
$$
   f\ =\ a_1X_1+\cdots+a_sX_s
$$
with  
$$
     a_i\in A\minus\zero \qquad\text{and}\qquad
     X_1>\cdots>X_s \quad\text{in }\ \XXe.
$$
Then we set  
$\lp(f)=X_1$, 
$\lc(f)=a_1$  and  
$\lt(f)=a_1X_1$; these are called the 
{\it leading power product} (or {\it leading monomial}),
{\it leading coefficient}  and  
{\it leading term} of  $f$, respectively.

%%%%%%%%%%%%%%%%%%%%%%%%%%%%%%%%%%%%%%%%%%%%%%%%%%%%%% f\ToF g の定義
\begin{definition}
\label{def:module-reduce}
Let  $f,h$  be two elements of  $L$  and  
$F=\{f_1,\ldots,f_s\}$  a finite subset of  $L\minus\zero$. 

\sn
(i) 
We write 
$f\toF h$ if 
there exist 
$a_1,\ldots ,a_s \in A$,  and 
$\x_1,\ldots,\x_s \in \XX$ satisfying 

\sn
\quad$\bullet$ 
$h=f-(a_1\x_1f_1 + \cdots + a_s \x_s f_s)$, 
\\
\quad$\bullet$ 
$\lp(f)=\x_i\lp(f_i)$  for all  $i$  such that  $a_i\neq 0$, and 
\\
\quad$\bullet$ 
$\lt(f)=a_1\x_1\lt(f_1)+\cdots+a_s\x_s\lt(f_s)$.  

\sn
(ii) 
We say that  $f$  {\it reduces to  $h$  modulo  $F$}, 
and write  $f \ToF h$,  
if there exist finitely many elements  
$h_1,\ldots,h_t\in L$  such that 
$$
  f\toF h_1 \toF h_2 \toF \cdots \toF h_t \toF h.
$$
We say that  $f$  reduces {\it strictly} to  $h$  if  
$f\ToF h$  and  $\lp(h)<\lp(f)$. 
\end{definition}

\begin{remarks} %%%%%%%%%%%%%%%%%%%%%%%%%%%%%%%%%%%%%%%%%% Remarks
(1)
If  $f \ToF h$  for  $f,h\in L$, 
then we have  $f\equiv g \mod \ideal{F}$, where  
$\ideal{F}$  denotes the $R$-submodule of  $L$  generated by  $F$. 

\sn
(2)
For any  $f$  and  $F$  as above, 
there exists a ``minimal reduction''  of  $f$  modulo  $F$, 
by which we mean an element  $f_0\in L$
to which  $f$  reduces modulo  $F$  
and which does not reduce strictly any more.  
This can be proved by induction on  $\lp(f)$, 
upon noticing the fact that the set  $\XXe$  is 
well-ordered with respect to  $<$.
\end{remarks}

For a subset  $G\subset L$, we denote by  
$\Lt(G)$  the submodule of  $L$  generated by the leading terms 
of all elements in  $G$.

The following theorem is fundamental for our purposes:

%%%%%%%%%%%%%%%%%%%%%%%%%%%%%%%%%%%%%%%%%%%%%%%%%%% 基本定理 on Ｇ基底
\begin{theorem}[\cite{Lez08}, Thm.\ 14, Cor.\ 15]
\label{thm:GBM}
Let  $M$  be a non-zero $R$-submodule of  $L$. 
Then there exists a finite subset  $G=\{g_1,\ldots,g_t\}$  of  
$M\minus\zero$  satisfying the following equivalent conditions:

\sn
$\mathrm{(a)}$ 
$\Lt(G) = \Lt(M)$.

\sn
$\mathrm{(b)}$ 
For any  $f\in L$, we have 
$f\in M$  if and only if  $f\ToG 0$.

\sn
$\mathrm{(c)}$ 
For any  $f\in M$, 
there exist  $h_1,\ldots,h_t\in R$  such that 
$f = h_1g_1+ \cdots + h_t g_t$  with 
$\lp(f) = \max_{1\le i \le t} (\lp(h_i) \lp(g_i))$. 
\end{theorem}

%%%%%%%%%%%%%%%%%%%%%%%%%%%%%%%%%%%%%%%%%%%%%%%%%%% 定義：Groebner基底 
\begin{definition}
\label{def:GB}
Let  $M$  be a non-zero $R$-submodule of  $L$.
A finite subset  $G$  of  $M\minus\zero$  as in the above theorem 
is called a {\it Gr\"obner basis for  $M$}.   
A finite subset  $G$  of  $L\minus\zero$  is called a 
{\it Gr\"obner basis} (in  $L$)  
if it is a Gr\"obner basis for some non-zero $R$-submodule of  $L$. 
\end{definition}

It follows immediately from the theorem that we have  
$M=\ideal{G}$  if  $G$  is a Gr\"obner basis 
for a non-zero $R$-submodule  $M$  in  $L$. 

A Gr\"obner basis is not unique for a given submodule. 
Indeed, if  $G$  is a Gr\"obner basis for  $M$, 
then any finite subset of  $M\minus\zero$  containing  $G$  
is again a Gr\"obner basis for  $M$ 
(this follows from (a) of Theorem \ref{thm:GBM}). 
In fact, even a minimal Gr\"obner basis (defined below) 
is not unique. 

%%%%%%%%%%%%%%%%%%%%%%%%%%%%%%%%%%%%%%%%%%%%%%% 定義：極小 Groebner 基底 
\begin{definition}[\cite{AL94}, Exer.\ 4.1.9]
\label{def:minimal}
A Gr\"obner basis  $G$  in  $L$  
is said to be {\it minimal}\/ 
if no  $g \in G$  can be strictly reduced with respect to  
$G \ssm \{g\}$.
\end{definition}

The minimality of a Gr\"obner basis  
$G=\{g_1,\ldots,g_t\}$  implies in particular that 
there exist no divisibility relations between 
the leading terms  $\lt(g_1),\ldots,\lt(g_t)$.  

Every Gr\"obner basis  
contains a minimal Gr\"obner basis. 
Indeed, 
if there is an element  $g\in G$  which can be strictly reduced 
with respect to  $G\ssm\{g\}$, we have  
$\lt(g)\in\Lt(G\ssm\{g\})$  (Def.\ \ref{def:module-reduce}). 
Hence  $\Lt(G)=\Lt(G\ssm\{g\})$, and 
$G\ssm\{g\}$  is also a Gr\"obner basis 
(\Cf Thm.\ \ref{thm:GBM}). 
Repeating this process finite times, we reach a 
minimal Gr\"obner basis.

In the rest of this section, 
let  $M$  be a non-zero $R$-submodule of  $L$. 
Choose a Gr\"obner basis  $G=\{g_1,\ldots ,g_t\}$  for  $M$. 
For each  $j$, write  $g_j=c_jX_j+\ldots$  with 
$c_j=\lc(g_j)$  and  
$X_j=\lp(g_j)$. 
For each  $X\in\XXe$, let  
$\IX$  denote the ideal of  $A$  generated by  
$c_j$  for all  $j$  such that  $X_j\mid X$  
(if there are no such  $j$, put  $\IX=0$), and choose 
a complete set of coset representatives  $C_X$  for  $A/\IX$ 
containing  $0$.  
For each  $X\in\XXe$, 
choose and fix an element  $\Xtil\in L$  of the form  
\begin{equation}\label{eq:Xtil} %%%%%%%%%%%%%%%%%%%%%%%%%%% eq:Xtil
   \Xtil\ =\ X+h\qquad\text{with}\quad \lp(h)<X.
\end{equation}
Then the family  $(\Xtil)_{X\in\XXe}$  
forms an $A$-basis of  $L$. 
Define the set  $\TG$  of {\it totally reduced} vectors by 
$$
  \TG\ :=\ \left\{\sum_{X\in\XXe}c_X\Xtil\ \Big|\ c_X\in C_X \right\}  
  \qquad\subset\ L.
$$
The set  $\TG$  depends on the choice of  $G$,   
$(C_X)_{X\in \XXe}$, and  $(\Xtil)_{X\in\XXe}$. 
We write this set  $\TG$  formally as 
$$
  \TG\ =\ \bigoplus_{X\in \XXe} C_X\Xtil.
$$

%%%%%%%%%%%%%%%%%%%%%%%%%%%%%%%%%%%%%%%%%%%%%%%%% 定理：加群版互除法
\begin{theorem}[\Cf\cite{AL94}, Thm.\ 4.3.3]
\label{thm:normal form}
The projection map  $L\to L/M$  induces a bijection of sets 
$$
   \rho:\ \TG\ \isom\ L/M.
$$
\end{theorem}

\begin{proof}
First we prove the injectivity of  $\rho$. 
Let  $f$  and  $h$  be two different elements of  $\TG$. 
Then any non-zero term of  $f-h$  has coefficient  
$\not\equiv 0\pmod{\IX}$, and hence  $f-h$  cannot be 
reduced strictly modulo  $G$  any more. 
By Theorem \ref{thm:GBM}, 
we have  $f-h\not\in M$.

Next we prove the surjectivity of  $\rho$. 
Given an  $f\in L$, we shall find an  $f_0\in\TG$  
such that  $f\ToG f_0$  by induction on  $\lp(f)$. 
This is trivial if  $f=0$. 
Suppose  $f\not=0$. Let  
$c=\lc(f)$  and  $X=\lp(f)$. 
Then there exists  $c_X\in C_X$  such that  
$c\equiv c_X$ (mod $I_X$). Write  
$c=c_X+\sum_jd_jc_j$  with  
$d_j\in A$, the sum being over those  $j$  
such that  $X_j\mid X$.
Write  $X=\x_jX_j$  with  $\x_j\in\XX$  for such  $j$. 
Then we have 
$$
   f\ =\ c_XX + \sum_jd_j\x_jg_j + f_1
$$
with some  $f_1\in L$  such that  
$\lp(f_1)<\lp(f)$. 
If  
$\Xtil=X+h$  with  $\lp(h)<X$, then 
$$
   f\ =\ c_X\Xtil + \sum_jd_j\x_jg_j + f_2
$$
with  $f_2:=f_1-c_Xh$. 
Since  $\lp(f_2)<\lp(f)$,  
the induction proceeds.
\end{proof}

%%%%%%%%%%%%%%%%%%%%%%%%%%%%%%%%%%%%%%%%%%%%%%%%%%%%%%%%%%%%%%%%
%%%%%%%%%%%%%%%%%%%%%%%%%%%%%%%%%%%%%%%%%%%%%%%%%%%%%% Section 3
%%%%%%%%%%%%%%%%%%%%%%%%%%%%%%%%%%%%%%%%%%%%%%%%%%%%%%%%%%%%%%%%
\section{Flatness of modules over a graded tdvr}
\label{sec:flat}
Let  $k$  be a field. 
In this section, we consider the case where 
the base ring  $A$  is 
a {\it graded tdvr}, 
by which we mean a graded ring  
$A=\oplus_{0\leq i<a}A_i$  
of length\footnote{%%%%%%%%%%%%%%%%%%%%%%%%%%%%%%%%%%%%%%% 脚注
Here and elsewhere, the direct sum  
$\oplus_{0\leq i<a}$  means the usual 
infinite direct sum if  $a=\infty$.
}
$a \le \infty$  such that the degree-$i$ submodule  $A_i$  is 
a one-dimensional vector space over $A_0=k$  
generated by  $(A_1)^i$. 
For example, the graded ring associated with a tdvr  
is a graded tdvr (see Sect.\ \ref{sec:tdvr}). 
By choosing a non-zero element  $\p\in A_1$, we may and do 
identify  $A$  with the truncated polynomial algebra  
$k[\p]/(\p^a)$  
(or the localized polynomial algebra  $k[\p]_{(\p)}$
if  $a=\infty$), and 
identify  $A_i$  with  $k\p^i$. 
Then any element  $x$  of  $A$  can be written as  
$x=u\p^i$  with  
$u\in A^\times$, and 
$0\leq i\leq a$
(with the convention  $0^0=1$  if  $a=1$). 
The exponent  $i\leq a$  is unique 
(we regard  $0=\p^\infty$  if  $a=\infty$); 
we denote it by  $v(x)$  and call it the {\it valuation} of  $x$. 
%Thus we have a map  
%$v:A\to\{0,\ldots,a\}$; we have  
%$v(x)=0$  if and only if  $x$  is a unit of  $A$, 
%$v(x)=1$  if and only if  $x$  is a uniformizer of  $A$, and 
%$v(x)=a$  if and only if  $x=0$.
Note that, if  $v(x)=i$, 
then we have  $ux\in A_i$  for some unit  $u\in A^\times$.

Let  $R$  be the polynomial ring  $A[x_1,\cdots,x_n]$; 
we regard it as a graded  $A$-algebra  
$R=\oplus_{0\leq i<a}R_i$  by setting  
$R_i=\oplus_{\x\in\XX}A_i\x$. 
Similarly, let  $L$  be the free $R$-module  
$Re_1\oplus\cdots\oplus Re_r$  and regard it as a graded  $R$-module  
$L=\oplus_{0\leq i<a}L_i$  by setting  
$L_i=R_ie_1\oplus\cdots\oplus R_ie_r$. 
We have  
$R_i=\p^iR_0$  and  
$L_i=\p^iL_0$.
Any  $g\in L$  can be written uniquely as  
$g=\suma\gi$  with  $\gi\in L_i$; 
we call  $\gi$  the {\it degree-$i$ part} of  $g$. 
A non-zero element  $g$  of  $L$  is said to be 
{\it homogeneous} of degree  $i$  if  $g=\gi$  
(\Ie if it is in  $L_i$). 
When we say  $M=\oplusa M_i$  is a graded $R$-submodule of  $L$, 
we assume that  $M_i=M\cap L_i$. 
A Gr\"obner basis for a non-zero graded $R$-submodule of  $L$  
is said to be {\it homogeneous} if every element of it 
is homogeneous. 

%%%%%%%%%%%%%%%%%%%%%%%%%%%%%%%%%%%%%%%%%%%%%%%% Lemma: Homogeneous GB
\begin{lemma}
Any non-zero graded  $R$-submodule of  $L$  
has a homogeneous Gr\"obner basis.
\end{lemma}

\begin{proof}
For any  $g\in L$, 
there exists a unit  $u$  of  $A$  such that  
$\lc(ug)\in A_i$  if  $i=v(\lc(g))$. 
If  $(ug)^{(i)}\in M_i$  is the degree-$i$ part of  $ug$, 
then we have  $\lt(ug)=\lt((ug)^{(i)})$. 
Now, if  $G$  is a Gr\"obner basis for  $M$, 
by (a) of Theorem \ref{thm:GBM}, 
we can replace each  $g\in G$  by  $(ug)^{(i)}$  as above  
to obtain a homogeneous Gr\"obner basis.
\end{proof}

Let  $M=\oplusa M_i$  be a non-zero graded $R$-submodule of  $L$, 
and fix a homogeneous Gr\"obner basis  
$G=\{g_1,\ldots,g_t\}$  for  $M$. 
For each  $j$, write  
$g_j = c_jX_j +\cdots$  with  
$c_j = \lc(g_j)$  and  
$X_j = \lp(g_j)$  as in the previous section. 
For each  $X\in\XXe$, let  
$\IX$  be the ideal of  $A$  generated by  
$c_j$  for all  $j$  such that  $X_j\mid X$. Denote by  
$m_X$  the length of  $A/\IX$  as an $A$-module. 
In other words, the ideal  $\IX$  is generated by  $\p^{m_X}$. 
Then we have 
\begin{equation}\label{eq:mX}
  m_X\ =\
    \begin{cases}
      \min\{v(c_j)\ |\ \text{$X_j$  divides  $X$}\},
         &\text{if there exists a  $j$  with  $X_j\mid X$}, \\
      a, &\mathrm{otherwise}.
    \end{cases}
\end{equation}
Set  
$A^{<m}:=\oplusm A_i$  
for each  $0\leq m\leq a$  
(if  $m=0$, then we put  $A^{<0}=\zero$); 
this is a $k$-subspace of  $A$, which, at the same time, 
we identify with the quotient 
$A/(\oplus_{i\geq m}A_i)$  
of  $A$  and then regard as a graded $A$-module.
For each  $X\in\XXe$, 
as the representative  $C_X$  for  $A/\IX$  in the 
previous section, we choose  
$C_X:=\AmX$.  
To construct a good set  $\TG$  of totally reduced vectors, 
we choose  $\Xtil=X+h$  as follows: 
$$
   \Xtil\ :=\ 
     \begin{cases}
       X         &\text{if  $m_X=0,a$},\\
       \x_jh_j   &\text{if  $m_X\not= 0,a$},
     \end{cases}
$$
where  
$j$  is an 
index such that  
$X_j\mid X$  and  $v(c_j)=m_X$, 
$\x_j$  is an element %%% \footnote{%%%%%%%%%%%%%%%%%%%%%%%%%% 脚注
% Caution: Do not confuse  $\x_j$  with  
% $x_j\in R=A[x_1,\ldots,x_n]$. 
of  $\XX$  such that  
$X=\x_jX_j$, and   
$h_j$  is an element of  $\XXe$  such that  
$g_j=c_jh_j$ 
(such  $\x_j$  exists because  $X_j\mid X$, and 
such  $h_j$  exists because  $g_j$  is homogeneous; 
the  $h_j$  is unique modulo $\m^{a-m_X}$). 
Note that this  $\Xtil$  has the form of (\ref{eq:Xtil}). 
The definition implies that  
$\p^{m_X}\Xtil\in M$. 
With these  $C_X$  and  $\Xtil$, 
we define
\begin{equation} %%%%%%%%%%%%%%%%%%%%%%%%%%%%%%%%%%%%%%%%% Eq: T_G
\label{eq:T_G}
   \TG\ :=\ \bigoplus_{X\in \XXe}\AmX\Xtil.
\end{equation}
This  $\TG$  depends not only on  $G$  but also on 
the choices made in finding  $\Xtil=\x_jh_j$  as above. 
It has a natural structure of graded $A$-module. 
It is flat as an $A$-module if and only if 
$m_X$  is either  $0$  or  $a$  for all  $X\in\XXe$.

%%%%%%%%%%%%%%%%%%%%%%%%%%%%%%%%%%%%%%%%%%%%%%%%% Lemma: T_G = L/M
\begin{lemma}\label{lem:T_G=L/M}
The map   
$\rho:\TG\isom L/M$  of Theorem \ref{thm:normal form} 
is an isomorphism of graded $A$-modules.
\end{lemma}

\begin{proof}
It is clear that  $\rho$  is compatible with 
addition and $k$-multiplication. 
It remains to check that it is compatible with 
multiplication by  $\p$. 
It is enough to check this on each  $\AmX\Xtil$. 
This is clear on  $A_i\Xtil$  if  $i<m_X-1$. 
If  $f\in A_{m_X-1}\Xtil$, then we have  
$\p f=0$  (because  $\p A_{m_X-1}=0$  in  
$\AmX=A/(\oplus_{i\geq m_X}A_i)$)  and  
$\p\rho(f)=0$  
(because  $\p^{m_X}\Xtil\in M$  by the definition of  $\Xtil$). 
\end{proof}

This together with the presentation  
(\ref{eq:T_G}) implies the equivalence of (a) and (b) 
in the next theorem:

%Note that the invariant  $m_X$  depends on the choice 
%of the Gr\"obner bases $G$. 

%%%%%%%%%%%%%%%%%%%%%%%%%%%%%%%%%%%%%%%%%%%% 定理：L/M の平坦性判定条件
\begin{theorem}
\label{thm:flat}
The following conditions 
on the $A$-module  $L/M$  are equivalent: 

\sn
$\mathrm{(a)}$ 
$L/M$  is flat over  $A$;

\sn
$\mathrm{(b)}$ 
$\{m_X|\ X\in\XXe\}\subset\{0,a\}$;

\sn
$\mathrm{(c)}$ 
For any  $j$, there exists  $j'$  such that 
$c_{j'}$  is a unit element in  $A$  and  $X_{j'}\mid X_j$.

\sn
Furthermore, we have 
$$
   \rank_A(L/M)\ =\ \#\{X\in\XXe\ |\ m_X = a\}
$$  
if  $L/M$  is flat over  $A$.
\end{theorem} 

%%%%%%%%%%%%%%%%%%%%%%%%% proof
\begin{proof}
It remains to show the equivalence of the conditions (b) and (c). 
Note that  $m_{X_j}$  cannot be  $a$  for any  $j$  
because  $c_{j'}\not=0$  for all  $j'$.
Then the assumption (b) implies that  
$m_{X_j}=0$  for any  $j$. 
This together with the definition of  $m_{X_j}$  implies  (c). 
Conversely, 
assume  (c)  and let  $X\in\XXe$. 
If  $m_X\not=a$, then we have  $X_j\mid X$  for some $j$. 
By the assumption, there exists  $j'$  such that 
$c_{j'}$  is a unit and  $X_{j'} \mid X_j$,   
and hence  $m_X=v(c_{j'})=0$. 
The statement on the rank follows from 
(\ref{eq:T_G}) and Lemma \ref{lem:T_G=L/M}.
\end{proof}

If we further assume that the Gr\"obner basis  $G$  is minimal, 
then there are no divisibility relations between the 
leading terms  $\lt(g_j)=c_jX_j$, and hence 
the condition (c) in the above theorem means that 
all  $c_j$  are units. 
Thus we deduce the following:

%%%%%%%%%%%%%%%%%%%%%%%%%%%%%%%%%%%%%%%%%%%%%%%% cor
\begin{corollary}
\label{cor:flat}
If the Gr\"obner basis  $G = \{g_1,\ldots ,g_t\}$  for  $M$  is minimal, 
then   $L/M$  is flat over  $A$  
if and only if  $c_j = \lc(g_j)$  is a unit element in  $A$  for every 
$1\le j\le t$. 
\end{corollary}

%%%%%%%%%%%%%%%%%%%%%%%%%%%%%%%%%%%%%%%%%%%%%%%%%%%%%%%%%%%%%%%%
%%%%%%%%%%%%%%%%%%%%%%%%%%%%%%%%%%%%%%%%%%%%%%%%%%%%%% Section 4
%%%%%%%%%%%%%%%%%%%%%%%%%%%%%%%%%%%%%%%%%%%%%%%%%%%%%%%%%%%%%%%%
\section{Flatness of modules over a tdvr}
\label{sec:tdvr}
First we recall some basic notions on tdvr's 
from \cite{HT1} and \cite{HT2}. 
%%%%%%%%%%%%%%%%%%%%%%%%%%%%%%%%%%%%%%%%%%%%%%%%% Review of tdvr
A {\it tdvr}\/ is an Artinian local ring whose maximal ideal is 
generated by one element. 
The {\it length}\/ of a tdvr  $A$  is the length of  
$A$  as an  $A$-module. 
If  $\O$  is a complete discrete valuation ring and  
$\m$  is its maximal ideal, 
then  $\O/\m^a$  is a tdvr for any integer  $a\geq 1$. 
Conversely, it is known that any tdvr is a quotient of a 
complete discrete valuation ring 
(\cite{HT1}, Prop.\ 2.2). 
A complete discrete valuation ring may naturally be thought of 
as a tdvr of length  $\infty$.
By abuse of terminology, however, 
we call also a discrete valuation ring which 
{\it may not necessarily be complete} 
a tdvr of length  $\infty$, because the theory below 
applies to any discrete valuation ring as well. 
It is known that a tdvr  $A$  is principal, 
and any ideal is of the form  $\m^i$  for some  $i\geq 0$, 
where  $\m$  is the maximal ideal of  $A$.  
Hence, the graded ring $\gr(A) = \oplusa \m^i/\m^{i+1}$  
is a graded tdvr of length equal to  $a=\length(A)$. 
 
%%%%% tdvr の付値は今のところ未使用
%Any generator  $\pi$  of the maximal ideal  $\m$  is said to be 
%a {\it uniformizer} of  $A$. 
%Any element  $x$  of a tdvr  $A$  of length  $a$  
%can be written as  
%$x=u\pi^i$  with  
%$u\in A^\times$, 
%$\pi$  a uniformizer of  $A$,  and  
%$0\leq i\leq a$
%(with the convention  $0^0=1$  if  $a=1$). 
%The exponent  $i\leq a$  is unique 
%(we regard  $0=\pi^\infty$  if  $a=\infty$); 
%we denote it by  $v(x)$  and call it the {\it valuation} of  $x$. 
%Thus we have a map  
%$v:A\to\{0,\ldots,a\}$; we have  
%$v(x)=0$  if and only if  $x$  is a unit of  $A$, 
%$v(x)=1$  if and only if  $x$  is a uniformizer of  $A$, and 
%$v(x)=a$  if and only if  $x=0$.

In this section, our base ring  $A$  is a tdvr; 
let  $\m$  be its maximal ideal and  $a$  its length. 
Let  $R$  be the polynomial ring  $A[x_1,\cdots,x_n]$,  and  
let  $L$  be the free $R$-module  
$Re_1\oplus\cdots\oplus Re_r$. 
Suppose we are given an $R$-submodule  $M$  of  $L$.
We are interested in the flatness over  $A$  of the 
quotient module  $L/M$. 
It is known 
(\cite{Bou98}, Chap.\ III, Sect.\ 5.2; \cite{Matsumura}, Thm.\ 22.3)
that  $L/M$  is flat over  $A$  if and only if 
its associated graded module  $\grLM$  is flat over  $\grA$, 
to which the results in the previous section can be applied. 
Here, the graded objects are defined by using the 
$\m$-adic filtration. Precisely stating, we define  
$\grA =\oplusa\griA$  and  
$\grLM=\oplusa\griLM$  with
$\griA :=\m^i/\m^{i+1}$   and   
$\griLM:=\m^i(L/M)/\m^{i+1}(L/M)
\simeq\m^iL/(M+\m^{i+1}L)$, 
respectively. 
The latter is a graded  $\grR$-module, where  
$\grR=\oplusa\griR$  is the graded $\grA$-algebra with
$\griR:=\m^iR/\m^{i+1}R$. 
Note that we have  
$\grLM\simeq\grL/\grM$, where  
$\grL=\oplusa\griL$  is the graded $\grR$-module with 
$\griL:=\m^iL/\m^{i+1}L$, and 
$\grM=\oplusa\griM$  is the graded  $\grR$-submodule of  $\grL$  
with  
$\griM:=(M\cap\m^iL)/(M\cap\m^{i+1}L)$. 
The set  $\XXe$  of power products in  $L$  may naturally be 
regarded as a $\grA$-basis of  $\grL$, so that we identify  
$\grL$  with  $\oplus_{X\in\XXe}\grA X$.

Now our criterion for the flatness of  $L/M$  can be 
stated as follows: 
Given an $R$-submodule  $M$  of  $L$, 
find a set of generators of  $\grM$, 
find a Gr\"obner basis  $G=\{g_1,\ldots,g_t\}$  of  $\grM$,
write  $g_j=c_jX_j+\cdots$  with  
$c_jX_j=\lt(g_j)$  and, using these  $g_j$,  
calculate the invariant  $m_X$  for each  $X\in\XXe$  
as in Section \ref{sec:flat} 
(see (\ref{eq:mX})). 
Then the non-graded version of Theorem \ref{thm:flat} 
holds as follows:

%%%%%%%%%%%%%%%%%%%%%%%%%%%%%%%%%%%%%%%%%%%% 定理：L/M の平坦性判定条件
\begin{theorem}
\label{thm:criterion}
The following conditions 
on the $A$-module  $L/M$  are equivalent: 

\sn
$\mathrm{(a)}$ 
$L/M$  is flat over  $A$;

\sn
$\mathrm{(b)}$ 
$\{m_X|\ X\in\XXe\}\subset\{0,a\}$;

\sn
$\mathrm{(c)}$ 
For any  $j$, there exists  $j'$  such that 
$c_{j'}$  is a unit element in  $\grA$  and  $X_{j'}\mid X_j$.

\sn
If the Gr\"obner basis  $G$  is minimal, 
these conditions are equivalent also to:

\sn
$\mathrm{(c')}$ 
The leading coefficient  $c_j$  of  $g_j$  is a unit element in  
$\grA$  for every  $1\le j\le t$. 

%\sn
%Furthermore, we have 
%$$
%   \rank_A(L/M)\ =\ \#\{X\in\XXe\ |\ m_X = a\}
%$$  
%if  $L/M$  is flat over  $A$.
\end{theorem}

%%%%%%%%%%%%%%%%%%%%%%%%%%%%%% remark
\begin{remark}
It is known that 
a flat module over a local ring with {\it nilpotent} maximal ideal 
is free (\cite{Matsumura}, Th.\ 7.10).  
Hence we have  
$$
   \rank_A(L/M)\ =\ \rank_{\grA}(\gr(L/M))
               \ =\ \#\{X\in\XXe\ |\ m_X = a\}
$$  
if  $L/M$  is flat over a tdvr  $A$  of {\it finite} length.
\end{remark}

A finite subset of $M\ssm \{0\}$ 
is said to be a {\it standard basis} for $M$ 
if its image in $\grM$ generates $\grM$. 
We have an algorithm to compute standard bases (\Cf \cite{Shibuta}).  
It outputs a set of generators of $\gr(M)$ if 
a set of generators of $M$ is given as input data. 
By running Buchberger's algorithm (\cite{AL94}, \cite{Lez08}), 
we obtain a Gr\"obner basis for $\grM$, 
and thus the above criterion is applicable.

\providecommand{\bysame}{\leavevmode\hbox to3em{\hrulefill}\thinspace}
\providecommand{\href}[2]{#2}

%%%%%%%%%%%%%%%%%%%%%%%%%%%%%%%%%%%%%%%%%%%%%%%%%%%%%%%%%%%%%
\bn
Toshiro Hiranouchi \\
Research Institute for Mathematical Sciences, Kyoto University, \\
Kyoto 606-8502 Japan,\\ 
Email address: {\tt hira@kurims.kyoto-u.ac.jp }

\mn
Yuichiro Taguchi \\
Graduate School of Mathematics, Kyushu University 33, \\
Fukuoka 812-8581, Japan \\
Email address: {\tt taguchi@math.kyushu-u.ac.jp}
\end{document}